\documentclass[preprint,12pt]{elsarticle}

\usepackage{amssymb}
\usepackage{amsthm}

\newcommand{\sysn}{\left\{\begin{array}{rcl}}
\newcommand{\sysk}{\end{array}\right.}

\newtheorem{theorem}{Theorem}[section]
\newtheorem{lemma}[theorem]{Lemma}

\theoremstyle{example}
\newtheorem{example}[theorem]{Example}

\newtheorem{proposition}[theorem]{Proposition}
\theoremstyle{definition}
\newtheorem{definition}[theorem]{Definition}

\newtheorem{corollary}[theorem]{Corollary}

\begin{document}

\begin{frontmatter}



\title{On separability of the functional space with the open-point and bi-point-open topologies}


\author{Alexander V. Osipov}

\ead{OAB@list.ru}

\address{Ural Federal
 University, Institute of Mathematics and Mechanics, Ural Branch of the Russian Academy of Sciences, 16,
 S.Kovalevskaja street, 620219, Ekaterinburg, Russia}

\begin{abstract}
In this paper we study the property of separability of functional
space $C(X)$ with the open-point and bi-point-open topologies. We
show that it is consistent with $ZFC$ that there is a set of reals
of cardinality the continuum such that a set $C(X)$ with the
open-point topology isn't a separable space. We also show in a set
model (the iterated perfect set model) that for every set of reals
$X$ a set $C(X)$ with bi-point-open topology is a separable space.
\end{abstract}


\end{frontmatter}



\section{Introduction}

The space $C(X)$ with the point-open topology (also known as the
topology of pointwise convergence) is denoted by $C_{p}(X)$. It
has a subbase consisting of sets of the form

$[x,V]^{+}=\{f\in C(X): f(x)\in V\}$,

where $x\in X$ and $V$ is an open subset of real line
$\mathbb{R}$. In paper \cite{amk} was introduced two new
topologies on $C(X)$ that we call the open-point topology and the
bi-point-open topology. The open-point topology on $C(X)$ has a
subbase consisting of sets of the form

$[U,r]^{-}=\{f\in C(X): f^{-1}(r)\bigcap U\neq\emptyset\}$,

where $U$ is an open subset of $X$ and $r\in \mathbb{R}$. The
open-point topology on $C(X)$ is denoted by $h$ and the space
$C(X)$ equipped with the open-point topology $h$ is denoted by
$C_{h}(X)$.

Now the bi-point-open topology on $C(X)$ is the join of the
point-open topology $p$ and the open-point topology $h$. It is the
topology having subbase open sets of both kind: $[x,V]^{+}$ and
$[U,r]^{-}$, where $x\in X$ and $V$ is an open subset of
$\mathbb{R}$, while $U$ is an open subset of $X$ and $r\in
\mathbb{R}$. The bi-point-open topology on the space $C(X)$ is
denoted by $ph$ and the space $C(X)$ equipped with the
bi-point-open topology $ph$ is denoted by $C_{ph}(X)$. One can
also view the bi-point-open topology on $C(X)$ as the weak
topology on $C(X)$ generated by the identity maps $id_{1}:
C(X)\mapsto C_{p}(X)$ and $id_{2}: C(X)\mapsto C_{h}(X)$.

 In \cite{amk} and \cite{amk1}, the separation and countability properties of these two
 topologies on $C(X)$ have been studied.

In \cite{amk} the following statements were proved.

\medskip

$\bullet$ $C_{h}(\mathbb{P})$ is separable. (Proposition 5.1.)

$\bullet$ If $C_{h}(X)$ is separable, then every open subset of
$X$ is uncountable. (Theorem 5.2.)

$\bullet$ If $X$ has a countable $\pi$-base consisting of
nontrivial connected sets, then $C_{h}(X)$ is separable. (Theorem
5.5.)

$\bullet$ If $C_{ph}(X)$ is separable, then every open subset of
$X$ is uncountable. (Theorem 5.8.)

$\bullet$ If $X$ has a countable $\pi$-base consisting of
nontrivial connected sets and a coarser metrizable topology, then
$C_{ph}(X)$ is separable. (Theorem 5.10.)

\medskip

 In the present paper, we will continue to study the separability of spaces $C_{h}(X)$ and
 $C_{ph}(X)$.

In this paper we use the following conventions. The symbols
$\mathbb{R}$, $\mathbb{P}$, $\mathbb{Q}$ and $\mathbb{N}$ denote
the space of real numbers, irrational numbers, rational numbers
and natural numbers, respectively. Recall that a dispersion
character $\Delta(X)$ of $X$ is the minimum of cardinalities of
its nonempty open subsets.

By {\it set of reals} we mean a zero-dimensional, separable
metrizable space every non-empty open set which has the
cardinality the continuum.

\section{Main results}

Note that if the space $C_{h}(X)$ is a separable space then
$\Delta(X)\geq\mathfrak{c}$. Really, if $A=\{f_{i}\}$ is a
countable dense set of $C_{h}(X)$ then for each non-empty open set
$U$ of $X$ we have $\bigcup f_{i}(U)=\mathbb{R}$. It follows that
$|U|\geq\mathfrak{c}$.

Also note that if the space $C_{ph}(X)$ is a separable space then
$C_{p}(X)$ is a separable space and $C_{h}(X)$ is a separable. It
follows that $X$ is a separable submetrizable (coarser separable
metric topology) space and $\Delta(X)=\mathfrak{c}$.

Note also that if the space $C_{h}(X)$ is a separable space then
any point $x\in X$ isn't $P$-point (point for which the family of
neighbourhoods is closed under countable intersections) of $X$.

\begin{definition} Let $X$ be a topological space. A set
$A\subseteq X$ will be called {\it $\mathcal{I}$-set}  if there is
a continuous function $f\in C(X)$ such that $f(A)$ contains an
interval $\mathcal{I}=[a,b]\subset \mathbb{R}$.
\end{definition}

It is easily seen that in Definition the set $\mathcal{I}=[a,b]$
 can be replaced by $\mathbb{C}=2^{\omega}$ or $\mathbb{P}=\omega^{\omega}$.

\medskip
It is known that there exists a subset $B\subset \mathbb{R}$ such
that no uncountable closed set of $\mathbb{R}$ is contained either
$B$ or $\mathbb{R}\setminus B$. Such a subset $B$ is called a
Bernstein set.

Marcin Kysiak (in personal correspondence) was seen following
lemma.

\begin{lemma} \label{lem1} Let $B$ be a Bernstein set and $U$ be an non-empty
open set in $B$. Then $U$ is $\mathcal{I}$-set.
\end{lemma}

\begin{proof} Let $D\subset \mathbb{R}\setminus B$ be a countable
dense subset of the real line and let $\{U_n : n\in \omega \}$ be
a countable topology base consisting of open intervals with
endpoints in $D$. For every $n\in\omega$ the set $U_n\setminus D$
is homeomorphic to the Baire space $\omega^{\omega}$, and hence it
is homeomorphic to its Cartesian square $\omega^{\omega}\times
\omega^{\omega}$; let $h_n : (U_n \setminus D) \rightarrow
\omega^{\omega}\times \omega^{\omega}$ be a homeomorphism. As
every uncountable perfect Polish space is a continuous image of
$\omega^{\omega}$, let us fix a continuous mapping $F$ from
$\omega^{\omega}$ onto $\mathbb{R}$. Let us define $g_n : (U_n
\setminus D) \rightarrow \mathbb{R}$ as $g_n = F \circ \pi_1\circ
h_n$, where $\pi_1: \omega^{\omega}\times \omega^{\omega}
\rightarrow \omega^{\omega}$ is the projection on the first
coordinate, i.e. $\pi_1(x,y)=x $ for $x,y\in \omega^{\omega}$. As
for every $x\in \omega^{\omega}$ the set $h_{n}^{-1}[\{x\}\times
\omega^{\omega}]$ contains a perfect set, we have $B\bigcap
h_{n}^{-1}[\{x\}\times \omega^{\omega}]\neq \emptyset$ and
consequently $h_n[B]\bigcap (\{x\}\times \omega^{\omega})\neq
\emptyset$  so $\pi_1\circ h_n[B]=\omega^{\omega}$, hence $F\circ
\pi_1\circ h_n[B]=\mathbb{R}$. We have shown that
$g_n[B]=\mathbb{R}$ for every $n\in \omega$, where $g_n$ is a
continuous function defined on an open interval $U_n$. As the
endpoints of $U_n$ do not belong to $B$, the function
$(g_n)\upharpoonright B$ can be easily extended to a continuous
function $f : B \rightarrow \mathbb{R}$ which is still onto
$\mathbb{R}$ by the property of $g_n$. Let now $U\subset
\mathbb{R}$ be a nonempty open set. As $\{U_n : n\in \omega \}$
was a topology base, there exists $n\in \omega$ such that $U_n
\subseteq U$. Then $f[U \bigcap B] = \mathbb{R}$.

\end{proof}

\begin{theorem}\label{th1}
Let $X$ be a Tychonoff space and  $C_{h}(X)$ be a separable space.
Then $X$ has a $\pi$-network consisting of $\mathcal{I}$-sets.

\end{theorem}

\begin{proof} Let set $A=\{f_i\}$ be a countable dense subset of $C_{h}(X)$. Suppose, contrary our claim, that there is non-empty open
set $U$ such that for any $f\in C(X)$ the set $f(U)$ don't
contains interval of real line. Note that if for every Cantor set
$\mathbb{C}$ holds $(\mathbb{R}\setminus f_1(U))\nsupseteq
\mathbb{C}$ then $f_1(U)$ is Bernstein set.

 By lemma ~\ref{lem1}, there is continuous function
$g\in C(f_1(U))$ such that $g(f_1(U))$ contains interval of real
line. This contradicts our assumption. It follow that there is
Cantor set $\mathbb{C}_{1}$ such that $f_1(U)\bigcap
\mathbb{C}_{1}=\emptyset$. For the set $f_{2}(U)$ we have that
there is Cantor set $\mathbb{C}_{2}$ such that
$\mathbb{C}_{2}\subseteq \mathbb{C}_{1}$ and $f_2(U)\bigcap
\mathbb{C}_{2}=\emptyset$. We can now proceed analogously to the
$f_i(U)$ for each $i>2$. As a result of the induction, we obtain
countable family of Cantor sets $\{\mathbb{C}_{i}\}_{i}$ such that
$\mathbb{C}_{i+1}\subseteq \mathbb{C}_{i}$ for each $i\in
\mathbb{N}$. Choose $r\in \bigcap_{i} \mathbb{C}_{i}$ we have
$f_{i}\notin [U,r]^{-}$ for each $i\in \mathbb{R}$, which
contradicts density of set $A$.

\end{proof}

\begin{theorem}\label{th2}
Let $X$ be a Tychonoff space with countable $\pi$-base, then the
following are equivalent.

\begin{enumerate}

\item $C_{ph}(X)$ is a separable space.

\item $X$ is separable submetrizable space and it has a countable
$\pi$-network consisting of $\mathcal{I}$-sets.

\end{enumerate}

\end{theorem}

\begin{proof} $(1) \Rightarrow (2)$. The map $id_{2}: C_{ph}(X)\mapsto
C_{h}(X)$ is continuous map, hence $C_{h}(X)$ is separable space.
By Theorem \ref{th1}, the space $X$ has a countable $\pi$-network
consisting of $\mathcal{I}$-sets. The map $id_{1}:
C_{ph}(X)\mapsto C_{p}(X)$ is continuous map, hence $C_{p}(X)$ is
separable space. It follow that $X$ is a separable submetrizable
space.

$(2) \Rightarrow (1)$. Let $S=\{S_i\}$ be a countable
$\pi$-network of $X$ consisting of  $\mathcal{I}$-sets. By
definition of $\mathcal{I}$-sets, for each $S_{i}\in S$ there is
the continuous function $h_{i}\in C(X)$ such that $h_{i}(S_{i})$
contains an interval $[a_{i},b_{i}]$ of real line. Consider a
countable set

$\{ h_{i,p,q}(x)=\frac{p-q}{a_{i}-b_{i}}*h_{i}(x)+p-
\frac{p-q}{a_{i}-b_{i}}*a_{i}\}$

of continuous functions on $X$, where $i\in \mathbb{N}$, $p,q\in
\mathbb{Q}$. Let $\beta=\{B_j\}$ be countable base of $(X,\tau_1)$
where $\tau_1$ is separable metraizable topology on $X$ because of
$X$ is separable submetrizable space. For each pair
$(B_{j},B_{k})$ such that $\overline{B_{j}}\subseteq B_{k}$ define
continuous functions

$$h_{i,p,q,j,k}(x)= \left\{
\begin{array}{rcl}
h_{i,p,q}(x) \,\, \, \, \, \, for \, \, \,  x\in B_{j}  \\
$\bf{0}$ \, \, \, \, \, \, \, \, \,  for \, \, \,  x\in X\setminus B_{k}.   \\
\end{array}
\right.
$$

 and for each $v\in \mathbb{Q}$

$$d_{j,k,v}(x)= \left\{
\begin{array}{rcl}

v  & for &  x\in B_{j}                              \\
$\bf{0}$ &  for &  x\in X\setminus B_{k}.   \\
\end{array}
\right.
$$

Let $G$ be the set of finite sum of functions $h_{i,p,q,j,k}$ and
$d_{j,k,v}$ where $i,j,k\in\mathbb{N}$ and $p,q,v\in \mathbb{Q}$.
We claim that the countable set $G$ is dense set of $C_{ph}(X)$.

 By proposition 2.2 in \cite{amk}, let

 $W=[x_1,V_1]^{+}\bigcap...\bigcap[x_m,V_m]^{+}\bigcap[U_1,r_1]^{-}\bigcap...\bigcap[U_n,r_n]^{-}$
be a base set of $C_{ph}(X)$ where $n,m\in\mathbb{N}$, $x_{i}\in
X$, $V_{i}$ is open set of $\mathbb{R}$ for $i\in \overline{1,m}$,
$U_{j}$ is open set of $X$ and $r_j\in \mathbb{R}$ for $j\in
\overline{1,n}$ and for $i\neq j$, $x_{i}\neq x_{j}$ and
$\overline{U_{i}}\bigcap \overline{U_{j}}=\emptyset$.

Fix points $y_j\in U_j$ for $j=\overline{1,n}$ and choose
$B_{s_l}\in \beta$ for $l=\overline{1,n+m}$ such that
$\overline{B_{s_{l_1}}}\bigcap \overline{B_{s_{l_2}}}=\emptyset$
for $l_1\neq l_2$ and $l_1,l_2\in \overline{1,n+m}$ and $x_i\in
B_{s_l}$ for $l\in \overline{1,m}$ and $y_j\in B_{s_l}$ for $l\in
\overline{m+1,n}$. Choose $B_{s'_l}\in \beta$ for $l\in
\overline{1,m}$ such that $x_i\in B_{s'_l}$ and
$\overline{B_{s'_l}}\subseteq B_{s_l}$ and choose $B_{s'_l}\in
\beta$ for $l\in \overline{m+1,n+m}$  such that
$y_j\in\overline{B_{s'_l}}\subseteq B_{s_l}$ where $l=j+m$.

Fix points $v_i\in (V_i\bigcap \mathbb{Q})$ for $i\in
\overline{1,m}$ and $p_j,q_j\in \mathbb{Q}$ such that
$p_j<r_j<q_j$ for $j=\overline{1,n}$.

Consider $g\in G$ such that

$g=d_{s'_{1},s_{1},v_1}+...+d_{s'_{m},s_{m},v_m}+h_{i_1,p_1,q_1,s'_{m+1},s_{m+1}}+...+h_{i_n,p_n,q_n,s'_{m+n},s_{n+m}}$
where $S_{i_k}\subset B_{s'_l}\bigcap U_k$ for $k=\overline{1,n}$
and $l=k+m$.

Note that $g\in W$. This proves theorem.

\end{proof}

\begin{corollary}
Let $X$ be a Tychonoff space with countable $\pi$-base, then the
following are equivalent.

\begin{enumerate}

\item $C_{h}(X)$ is a separable space.

\item $X$  has a countable $\pi$-network consisting of
$\mathcal{I}$-sets.

\end{enumerate}

\end{corollary}

\begin{corollary}\label{th2}
If $X$ is a separable metrizable space, then the following are
equivalent.

\begin{enumerate}

\item $C_{ph}(X)$ is a separable space.

\item $X$ has a countable $\pi$-network consisting of
$\mathcal{I}$-sets.

\end{enumerate}

\end{corollary}

\begin{theorem}\label{th3}
If $X$ is a Tychonoff space with network consisting non-trivial
connected sets, then the following are equivalent.

\begin{enumerate}

\item $C_{ph}(X)$ is a separable space.

\item $X$ is a separable submetrizable space.

\end{enumerate}

\end{theorem}

\begin{proof}  $(1) \Rightarrow (2)$.  The map $id_{1}:
C_{ph}(X)\mapsto C_{p}(X)$ is a continuous map, hence $C_{p}(X)$
is a separable space. It follow that $X$ is a separable
submetrizable space.

 $(2) \Rightarrow (1)$.
Let $X$ be a separable submetrizable space,
  i.e. $X$ has coarser separable metric topology $\tau_1$ and $\gamma$ be network of $X$ consisting non-trivial
connected sets. Let $\beta=\{B_i\}$ be a countable base of
$(X,\tau_1)$. We can assume that $\beta$ closed under finite union
of its elements.

For each finite family $\{B_{s_i}\}_{i=1}^d\subset \beta$ such
that $\overline{B_{s_i}}\bigcap \overline{B_{s_j}}=\emptyset$ for
$i\neq j$ and $i,j\in \overline{1,d}$ and $\{p_i\}_{i=1}^d \subset
\mathbb{Q}$ we fix $f=f_{s_1,...,s_d,p_1...,p_d}\in C(X)$ such
that $f(\overline{B_{s_{i}}})=p_i$ for each $i=\overline{1,d}$.


Let $G$ be the set of functions $f_{s_1,...,s_d,p_1...,p_d}$ where
$s_i\in\mathbb{N}$ and $p_i\in \mathbb{Q}$ for $i\in\mathbb{N}$.
We claim that the countable set $G$ is dense set of $C_{ph}(X)$.

 By proposition 2.2 in \cite{amk}, let

 $W=[x_1,V_1]^{+}\bigcap...
\bigcap[x_m,V_m]^{+}\bigcap[U_1,r_1]^{-}\bigcap...\bigcap[U_n,r_n]^{-}$
be a base set of $C_{ph}(X)$ where $n,m\in\mathbb{N}$, $x_{i}\in
X$, $V_{i}$ is open set of $\mathbb{R}$ for $i\in \overline{1,m}$,
$U_{j}$ is open set of $X$ and $r_j\in \mathbb{R}$ for $j\in
\overline{1,n}$ and for $i\neq j$, $x_{i}\neq x_{j}$ and
$\overline{U_{i}}\bigcap \overline{U_{j}}=\emptyset$.

Choose $B_{s_l}\in \beta$ for $l=\overline{1,n+m}$ such that
$\overline{B_{s_{l_1}}}\bigcap \overline{B_{s_{l_2}}}=\emptyset$
for $l_1\neq l_2$ and $l_1,l_2\in \overline{1,n+m}$ and $x_i\in
B_{s_l}$ for $l\in \overline{1,m}$ and $B_{s_l}\bigcap U_{k}\neq
\emptyset$ for $l\in \overline{m+1,n+m}$ and $k=l-m$. Choose
$B_{s'_l}\in \beta$ for $l\in \overline{1,m}$ such that $x_i\in
B_{s'_l}$ and $\overline{B_{s'_l}}\subseteq B_{s_l}$  and choose
$A_{k}\in \gamma$ for $k\in \overline{1,m}$  such that
$A_{k}\subseteq (U_l\bigcap B_{s_l})$ where $l=k+m$.

Choose different points $s_{k}, t_{k}\in A_{k}$ for every
$k=\overline{1,m}$. Let $S,T\in \beta$ such that
$\overline{S}\bigcap \overline{T}=\emptyset$,
$\overline{B_l}\bigcap \overline{S}=\emptyset$,
$\overline{B_l}\bigcap \overline{T}=\emptyset$ for $l\in
\overline{1,m}$ and $s_{k}\in S$ and $t_{k}\in T$ for all
$k=\overline{1,m}$.

Fix points $v_i\in (V_i\bigcap \mathbb{Q})$ for $i\in
\overline{1,m}$.

Choose $p,q\in \mathbb{Q}$ such that $p<\min\{r_i :
i=\overline{1,n}\}$ and $q>\max\{r_i : i=\overline{1,n}\}$.

Let $$f(x)= \left\{
\begin{array}{rcl}

p  & for &  x\in \overline{S}                              \\
q &  for &  x\in \overline{T}                 \\
v_l & for & x\in \overline{B_{s'_l}} \\
\end{array}
\right.
$$

where $l\in \overline{1,m}$.

Note that $f\in W$. This proves theorem.

\end{proof}

\begin{theorem}\label{th4}
If $X$ is a locally connected space without isolated points, then
the following are equivalent.

\begin{enumerate}

\item $C_{ph}(X)$ is a separable space.

\item $X$ is a separable submetrizable space.

\end{enumerate}

\end{theorem}

\begin{definition} Let $(X,\tau)$ be a topological space. Define a cardinal function
$\xi(X)=min\{|\gamma|:$ for every finite family of pairwise
disjoint nonempty open subsets $\{V_i\}_{i=1}^{k}$ of $X$ there is
family of pairwise disjoint nonempty zero-sets
$\gamma'=\{Z_i\}_{i=1}^{k}\subseteq \gamma$ such that
$V_{i}\bigcap Z_{i}\neq \emptyset$ for $i=\overline{1,k}\}$.

\end{definition}

Obviously, that $\xi(X)\leq \pi w(X)$.

\begin{theorem}\label{th3}
If $X$ is a locally connected space without isolated points, then
the following are equivalent.

\begin{enumerate}

\item $C_{h}(X)$ is a separable space.

\item $\xi(X)=\aleph_0$.

\end{enumerate}

\end{theorem}

\begin{proof} $(1) \Rightarrow (2)$. Let $A=\{f_i\}$ be a
countable dense set of $C_{h}(X)$ and $\beta=\{B_j\}$ be a
countable base of $\mathbb{R}$. Define a family
$\gamma=\{f_i^{-1}(\overline{B_j}): i,j\in \mathbb{N} \}$.

Let $\{V_s\}_{s=1}^k$ be a finite family of pairwise disjoint
nonempty open subsets of $X$. Consider an
 open base set $W=[V_1,1]^{-}\bigcap...\bigcap[V_k,k]^{-}$. Then
 there are $f_{i'}\in A\bigcap W$ and the family $\{B_{j_s}: s\in B_{j_s}$ for $s\in  \overline{1,k}$  and $\overline{B_{j_s'}}\bigcap \overline{B_{j_s^{''}}}=\emptyset$
  for $s'\neq s^{''}$ and $s',s^{''}\in \overline{1,k}\}$ such that
 $\gamma'=\{ f_{i'}^{-1}(\overline{B_{j_s}})\}_{s=1}^{k}$ required the subfamily of $\gamma$.

$(2) \Rightarrow (1)$. Let $\gamma=\{F_{i}\}$ be family of
zero-sets from definition of $\xi(X)$ such that
$|\gamma|=\xi(X)=\aleph_0$. We can assume that $\gamma$ closed
under finite union of its elements. Consider countable set of
continuous functions

 $A=\{f_{i,j,p,q}\in C(X) : f_{i,j,p,q}(F_i)=p$ and
$f_{i,j,p,q}(F_j)=q$ for $F_i,F_j\in \gamma$ such that
$F_{i}\bigcap F_{j}=\emptyset$ and $p,q\in \mathbb{Q} \}$.

Let $W=[U_1,r_1]^{-}\bigcap...\bigcap[U_n,r_n]^{-}$ be a base set
of $C_{h}(X)$ where $n\in\mathbb{N}$, $U_{j}$ is open set of $X$
and $r_j\in \mathbb{R}$ for $j\in \overline{1,n}$ and for $i\neq
j$, $\overline{U_{i}}\bigcap \overline{U_{j}}=\emptyset$.

Fix connected open sets $S_{\alpha_i}$ such that
$S_{\alpha_i}\subset U_{i}$ for $i=\overline{1,n}$. Since
$S_{\alpha_i}$ is not-trivial set there are different points $a_i,
b_i\in S_i$ for $i=\overline{1,n}$. Let $\{O_i\}_{i=1}^{n}$ and
$\{O^i\}_{i=1}^n$ be families  of pairwise disjoint nonempty open
subsets of $X$ such that $a_i\in O_i$, $b_i\in O^i$ and
$(\bigcup_{i=1}^{n} O_i)\bigcap (\bigcup_{i=1}^{n}
O^i)=\emptyset$. There are family of pairwise disjoint nonempty
zero-set sets $\gamma'=\{F_{k}\}_{k=1}^{2n}\subset \gamma$ such
that $F_k\bigcap O_i\neq \emptyset$ for $k=i$  and $F_k\bigcap
O^i\neq \emptyset$ for $k=i+n$. Let $H_1=\bigcup_{k=1}^n F_k$ and
$H_2=\bigcup_{k=n+1}^{2n} F_k$, then consider $f=f_{i',j',p,q}\in
A$ such that $f(F_i')=p$ and $f(F_j')=q$ where $F_i'=H_1$,
$F_j'=H_2$ for some $i',j'\in \mathbb{N}$ and $p,q\in \mathbb{Q}$
such that $p<\min\{r_i : i=\overline{1,n}\}$ and $q>\max\{r_i :
i=\overline{1,n}\}$. Note that $f\in W$. This proves theorem.

\end{proof}

\section{Consistent counter examples}

Recall that a set of reals $X$ is {\it null} if for each positive
$\epsilon$ there exists a cover $\{I_{n}\}_{n\in \mathbb{N}}$ of
$X$ such that $\sum_n diam(I_n)< \epsilon$. A set of reals $X$ has
{\it strong measure zero} if, for each sequence
$\{\epsilon_n\}_{n\in \mathbb{N}}$ of positive reals, there exists
a cover $\{I_{n}\}_{n\in \mathbb{N}}$ of $X$ such that
$diam(I_n)<\epsilon_n$ for all $n$. For example, every Lusin set
has strong measure zero.

\begin{example}$(\bf CH)$ Let $X$ be a set of reals and it has strong measure zero.
Consider a space $C_{h}(X)$. Note that the property {\it has
strong measure zero} is invariant with respect to continuous
mappings \cite{kur}. Let $A=\{f_i\}_i\subset C(X)$ be a countable
set of continuous functions and $X_i=X$ for each $i\in
\mathbb{N}$. Direct sum $Y=\bigoplus_i X_i$ has strong measure
zero. Hence a set $F(Y)\subset \mathbb{R}$ has strong measure zero
where $F$ is a continuous real-valued function on $Y$. So if
$F\upharpoonright X_i=f_i$ we have that $\bigcup_i f_i(X)\neq
\mathbb{R}$. It follows that $C_{h}(X)$ ( a fortiori $C_{ph}(X)$ )
isn't a separable space.

\end{example}

 In \cite{mill} was shown that it is consistent with ZFC that for any set of
 reals of cardinality the continuum, there is a (uniformly) continuous map
 from that set onto the closed unit interval. In fact, this holds
 in the iterated perfect set model.

\begin{theorem}( the iterated perfect set model)\label{th20}

If $X$ is a separable metrizable space, then  the following are
equivalent.

\begin{enumerate}

\item $C_{ph}(X)$ is a separable space.

\item  $\Delta(X)=\mathfrak{c}$.

\end{enumerate}

\end{theorem}

\begin{proof} $(2) \Rightarrow (1)$. Note that in the iterated perfect set model
every nonempty open set of $X$ is a $\mathcal{I}$-set. Really,
suppose that $U$ is a nonempty open set of $X$, but it isn't a
$\mathcal{I}$-set. Then $U$ is a set of reals of cardinality the
continuum. Note that for each point $x\in U$ there exist
continuous function $f:X\rightarrow \mathcal{I}$ such that
$f^{-1}(0) \supseteq X\setminus U$ and $f^{-1}(1)\ni x$. Clearly,
that there is $r\in \mathcal{I}$ such that $r\notin f(U)$. It
follows that $f^{-1}([r,1])$ is clopen neighborhood of $x$ and $U$
is a zero-dimensional subspace of $X$. Let $W$ be an open set such
that $\overline{W}\subset U$. Then $\overline{W}$ is a set of
reals of cardinality the continuum. By the iterated perfect set
model there exist continuous function $h$ from $\overline{W}$ onto
the closed unit interval $\mathcal{I}$. Therefore, from
Tietze-Urysohn Extension Theorem, there is a continuous function
$F:X\rightarrow \mathbb{R}$ such that $F\upharpoonright
\overline{W}=h$ and $F(U)\supseteq \mathcal{I}$. This contradicts
our assumption.

 By Theorem ~\ref{th2}, $C_{ph}(X)$ is a separable
space.

\end{proof}

\section{Remarks}

Now we state two results of \cite{amk1} that give for the density
of the spaces $C_{h}(X)$ and $C_{ph}(X)$.

1.(\cite{amk1}, Theorem 4.21)\label{th10} If $X$ is a locally
connected space with no isolated points, then $d(C_{h}(X))= \pi
w(X)$.

2.(\cite{amk1}, Theorem 4.22)\label{th15} If $X$ is a locally
connected space with no isolated points, then $d(C_{ph}(X))= \pi
w(X)\cdot iw(X)$.

\medskip
We note that these results are false (equality can not be !), but
in these results meaning an upper bound for the density of the
spaces $C_{h}(X)$ and $C_{ph}(X)$.

\begin{theorem}
If $X$ is a locally connected space with no isolated points, then
$d(C_{h}(X))\leq \pi w(X)$.
\end{theorem}

\begin{theorem} If $X$ is a locally connected space with no isolated points, then
$d(C_{ph}(X))\leq \pi w(X)\cdot iw(X)$.

\end{theorem}

Now we give an example where there is no equality.

\begin{example} Let $X=\oplus_{\alpha<\mathfrak{c}} \mathbb{R}_{\alpha}$ be
a direct sum of real lines $\mathbb{R}$. Then $X$ is a separable
submetrizable space i.e. $iw(X)=\aleph_0$. Clearly, that $\pi
w(X)=\mathfrak{c}$. By Theorem $\ref{th4}$, $C_{ph}(X)$ is
separable, and, hence, $C_{h}(X)$ is separable.

\end{example}

\begin{proposition} If $C_{h}(X)$ is a separable space, then
$C_{h}(\beta X)$ is a separable space.

\end{proposition}

\begin{proof}
Note that $C_{h}(X)$ is homeomorphic to $C_{h}(X,(0,1))$. Let
$A=\{f_{i}\}$ be a countable dense set of $C_{h}(X,(0,1))$. Then
set $\{\widetilde{f_i}\}$ is countable dense subset of
$C_{h}(\beta X,(0,1))$ where $\widetilde{f_i}\upharpoonright
X=f_{i}$. Really let
$W=[U_1,r_1]^{-}\bigcap...\bigcap[U_n,r_n]^{-}$ be a base set of
$C_{h}(\beta X)$ where $n\in\mathbb{N}$, $U_{j}$ is open set of
$\beta X$ and $r_j\in \mathbb{R}$ for $j\in \overline{1,n}$ and
for $i\neq j$, $\overline{U_{i}}\bigcap
\overline{U_{j}}=\emptyset$. Clearly that
$V=[V_1,r_1]^{-}\bigcap...\bigcap[V_n,r_n]^{-}$ be a open set of
$C_{h}(X)$ where $n\in\mathbb{N}$, $V_{j}=X\bigcap U_{j}$ is open
set of $X$ and $r_j\in \mathbb{R}$ for $j\in \overline{1,n}$ and
for $i\neq j$, $\overline{V_{i}}\bigcap
\overline{V_{j}}=\emptyset$. There is $f_i'\in A\bigcap V$ and it
follows that $\widetilde{f_i'}\in W$.

\end{proof}

\begin{example} Let $X=\mathbb{R}$. By Theorem \ref{th2},
$C_{ph}(X)$ is a separable space, but $C_{ph}(\beta X)$ is not a
separable space because  $\beta X$ is not a separably
submetrizable space.

\end{example}

\medskip

\section{Acknowledgement}

This work was supported by Act 211 Government of the Russian
Federation, contract ¹ 02.A03.21.0006.

\bibliographystyle{model1a-num-names}
\bibliography{<your-bib-database>}







\end{document}